\documentclass[10pt]{amsart}
\usepackage[british,UKenglish,USenglish,english,american]{babel}
\usepackage[utf8]{inputenc}
\usepackage{braket}
\usepackage{xfrac}
\usepackage{amsfonts, amssymb, amsmath, amsthm, amscd}
\usepackage{mathtools}
\usepackage{tikz}
\usepackage{empheq}
\usepackage[left=2.2cm,right=2.2cm,top=2.2cm,bottom=2.2cm]{geometry}
\usepackage{color}
\usepackage{epsfig}  	
\usepackage{epic,eepic}   

\DeclarePairedDelimiterX{\norm}[1]{\lVert}{\rVert}{#1}
\newtheorem{thm}{Theorem}[section]
\newtheorem*{mainthm}{Main Theorem}
\newtheorem{prop}[thm]{Proposition}
\newtheorem{lem}[thm]{Lemma}
\newtheorem{cor}[thm]{Corollary}

\theoremstyle{definition}
\newtheorem{definition}[thm]{Definition}
\newtheorem{example}[thm]{Example}

\theoremstyle{remark}
\newtheorem{remark}[thm]{Remark}
\DeclareMathOperator{\vol}{vol}

\begin{document}

\begin{abstract}
We consider closed and orientable immersed hypersurfaces of translational manifolds. Given a vector field on such a hypersurface, we define a perturbation
of its Gauss map, which allows us to obtain topological invariants for the immersion that depends on the geometry of the manifold and on the ambient space.
We use these quantities to find obstructions to the existence of certain codimension-one foliations and distributions.
\end{abstract}

\title{Topological invariants for closed hypersurfaces}

\author{\'{I}caro Gon\c{c}alves}
\author{Eduardo Longa}

\address{Departamento de Matem\'{a}tica, Instituto de Matem\'{a}tica e Estat\'{i}stica, Universidade de S\~{a}o Paulo, R. do Mat\~{a}o 1010,
S\~{a}o Paulo, SP 05508-900, Brazil}
\email{icarog@ime.usp.br}

\address{Departamento de Matem\'{a}tica, Instituto de Matem\'{a}tica e Estat\'{i}stica, Universidade Federal do Rio Grande do Sul, Av. Bento Gon\c{c}alves 9500, 
Porto Alegre, RS 91509-900, Brazil} 
\email{eduardo.longa@ufrgs.br}



\maketitle

\vspace{-0.5cm}

\begin{center}
\emph{Dedicated to Professor Fabiano Brito, on the occasion of his $70^{\text{th}}$ birthday}.
\end{center}

\bigskip
\let\thefootnote\relax\footnote{\'{I}. Gon\c{c}alves was partially supported by a scholarship from CNPq (Brazil), 141113/2013-8 and by the National Postdoctoral Program/CAPES  (Brazil).
 E. Longa was partially supported by a scholarship from CNPq (Brazil), 130653/2015-2.}

\section{Introduction}

The study of hypersurfaces in Riemannian manifolds constitutes a central feature in differential geometry. Their extrinsic aspects have branched out into new
concepts and tools and have been instigating profound questions concerning the geometry of submanifolds over the last decades. Regarding how the geometry of a 
Riemannian manifold relates to its topology, closed Euclidean hypersurfaces played an initial role towards building the concepts behind Gauss-Bonnet-Chern 
theorem. In this framework, translating geometry to topology can be achieved through the Gauss map of the hypersurface. The determinant of its Jacobian matrix
provides the geometry, while for the even dimensional case its degree equals half the Euler characteristic of the manifold. Although the degree no longer relates
to the Euler number when the dimension is odd, it remains to be an important tool in the theory. For instance, a direct consequence of a theorem of Hopf is that
this degree characterizes homotopic immersions, and Milnor's results in \cite{milnor56} exhibits what are the possible values it can assume. 

A natural extension is to consider Riemannian manifolds which admit sufficient structure in order to have a proper definition of a Gauss map for its
hypersurfaces. A major step in this direction was taken by Ripoll in \cite{ripoll91}. The author examined hypersurfaces of Lie groups equipped with a left
invariant metric, the Gauss map being the left translation of the unit normal into the unit sphere of the Lie algebra. Recently, the second author and Ripoll
\cite{longa2016} introduced the notion of a translational Riemannian manifold, where to each point is assigned an isometry mapping the tangent space at this 
point to a given vector space. Thus, the notion of a Gauss map extends to hypersurfaces in this broader class of manifolds.


On the other hand, unitary vector fields on closed Euclidean hypersurfaces together with the normal map are the fundamental ingredients for constructing new 
integral formulae for distributions on these hypersurfaces, \cite{brito2016}. When compared to the main integrals for foliations and distributions found in the 
literature, e.g. \cite{andrzejewski2010}, \cite{brito81}, \cite{rovenski2012}, \cite{rovenski2011} and references therein, these invariants exhibit a new behaviour: they depend on 
the topology of the foliated space by means of the degree of the Gauss map. In the present article, we obtain similar topological invariants for hypersurfaces
immersed in translational Riemannian manifolds. We prove the following.

\begin{mainthm} Let $M^{n+1}$ be a closed, oriented and immersed hypersurface of a translational Riemannian manifold $\left( \overline{M}, \Gamma \right)$, with unit
normal $\eta$. Suppose $\chi(M) = 0$. For a unitary vector field $v$ on $M$ and positive $t$, consider the map $\varphi_t^v: M \to V$ defined by
\begin{align*}
\varphi_t^v(p) = \Gamma_p(\eta(p) + tv(p)), \quad p \in M.
\end{align*}
Furthermore, define functions $\mu_k : M \to \mathbb{R}$ by 
\begin{align*}
\det(D \varphi_t^v) = \sqrt{1+t^2} \sum_k \mu_k t^k.
\end{align*}
Then, the following formula holds:
\begin{align}
\int_M \mu_k = 
\begin{cases}
\deg(\gamma) \vol(\mathbb{S}^{n+1}) \binom{n/2}{k/2}, &\quad \text{if } n \text{ and } k \text{ are even} \\ 
0, &\quad \text{if } n \text{ or } k \text{ is odd}  
\end{cases} \label{curvature integrals}
\end{align}
where $\gamma : M \to \mathbb{S}^{n+1}$ is the Gauss map associated to $\eta$.
\end{mainthm}

The last section will be devoted to some applications of the above theorem to the theory of foliations, and some obstructions to certain distributions as well. 

\section{Translational manifolds an the Gauss map} \label{The map}

The hypersurfaces will be lying in a translational ambient, according to the following definition.

\begin{definition} A translational Riemannian manifold consists of a pair $\left( \overline{M}, \Gamma \right)$, where $\overline{M}$ is an $(n+2)$-dimensional 
Riemannian manifold, $\Gamma : T \overline{M} \to \overline{M} \times V$ is a smooth vector bundle map, and $V$ is an $(n+2)$-dimensional real vector space with 
an inner product such that the map $\Gamma_p : T_p \overline{M} \to V$ implicitly defined by
\begin{align*}
v \mapsto \Gamma(p,v) = \left( p,\Gamma_{p}(v) \right)
\end{align*} 
is an isometry for every point $p \in \overline{M}$. When $\Gamma$ is understood from context, we simply say $\overline{M}$ is a translational manifold.
\end{definition}

Throughout this paper, a translational manifold $\left( \overline{M}, \Gamma \right)$ will be fixed. Let $M^{n+1}$ be a compact, connected and orientable manifold 
and let $f : M^{n+1} \to \overline{M}$ be an immersion of $M$ into $\overline{M}$, with a normal unitary vector field $\eta : M \to T\overline{M}|_M$. We identify 
small neighbourhoods of $M$ and their images via $f$ and the tangent spaces to $M$ with their images via $Df$. Denote by $\mathbb{S}^{n+1}(r)$ the sphere of radius $r$ 
of $V$, centered at the origin; $\mathbb{S}^{n+1}$ will denote the unit sphere of $V$.

\begin{definition} The Gauss map $\gamma : M \to \mathbb{S}^{n+1}$ associated to the normal vector field $\eta$ is given by
\begin{align*}
\gamma(p) = \Gamma_p(\eta(p)), \quad p \in M.
\end{align*}
\end{definition}

Next, we define a special type of vector field that will play an important role.

\begin{definition} Given a vector $X \in T_p \overline{M}$, the vector field $\widetilde{X} \in \mathfrak{X}(\overline{M})$ defined by
\begin{align*}
\widetilde{X}(q) = \left( \Gamma_q^{-1} \circ \Gamma_p \right)(X), \quad q \in \overline{M}
\end{align*}
is called the $\Gamma$-invariant (or simply invariant) vector field of $\overline{M}$ associated with $X$.
\end{definition}

\begin{example}[Left translation on Lie groups] \label{Left translation} Let $\overline{M} = G$ be a Lie group and $V = \mathfrak{g}$ be the Lie algebra of $G$,
considered as the tangent space of $G$ at the identity. Choose a left invariant metric for $G$ and define $\Gamma : TG \to G \times \mathfrak{g}$ by
\begin{align*}
\Gamma(g, v) = \left( g, DL_{g^-1}(g) \cdot v \right), \quad (g, v) \in TG,
\end{align*}
where $L_x : y \mapsto xy$ is the left translation. Here, the $\Gamma$-invariant vector fields are the left invariant vector fields of $G$. This is the setting 
studied in \cite{ripoll91}. 
\end{example}

\begin{example}[Parallel transport] \label{Parallel transport} Assume $\overline{M}$ is a Cartan-Hada\-mard manifold, that is, a complete, connected and simply connected
Riemannian manifold with nonpositive sectional curvature. Given a distinguished point $p_0\in \overline{M}$, the exponential map at $p_0$ is, by Hadamard's Theorem, a 
diffeomorphism from $T_{p_0} \overline{M}$ onto $\overline{M}$, so that every point on the manifold can be joined to $p_0$ by a unique geodesic. Setting 
$V = T_{p_0} \overline{M}$, we may then define $\Gamma_p : T_p \overline{M} \to V$ by choosing $\Gamma_p(v)$ as being the parallel transport of $v \in T_p \overline{M}$ to 
$T_{p_0}\overline{M}$ along this geodesic. Thus, the invariant vector fields here are the parallel vector fields along the geodesic rays issuing from $p_0$. 

More generally, given any complete Riemannian manifold $\overline{M}$ and a point $p_0$ in $\overline{M}$, we can define the parallel transport to $T_{p_0} \overline{M} $ on 
$\overline{M} \setminus C_{p_0}$ as above, where $C_{p_0}$ is the cut locus of $p_0$.
\end{example}

Let $\overline{\nabla}$ be the Levi-Civita connection of $\overline{M}$. Recall that the shape operator of $M$ is the section
$A$ of the vector bundle $\operatorname{End}(TM)$ of endomorphisms of $TM$ given by
\begin{align*}
A_p(X) = - \overline{\nabla}_X \eta, \quad (p, X) \in TM.
\end{align*}

Similarly, we define another section of $\operatorname{End}(TM)$, which depends additionally on the choice of the translation $\Gamma$.

\begin{definition} \label{invariant shape operator} The invariant shape operator of $M$ is the section $\alpha$ of the bundle $\operatorname{End}(TM)$ given by
\begin{align*}
\alpha_p(X) = \overline{\nabla}_X \widetilde{\eta(p)}, \quad (p, X) \in TM.
\end{align*}
\end{definition}

The proposition below establishes a relationship between $\gamma$ and the extrinsic geometry of $M$.

\begin{prop} \label{relationship} For any $p \in M$, the following identity holds:
\begin{align*}
\Gamma_p^{-1} \circ D \gamma(p) = - \left( A_p + \alpha_p \right).
\end{align*}
\end{prop}
\begin{proof}
Fix $p \in M$ and an orthonormal basis $\{e_1, \dots, e_{n+2} \}$ of $T_p \overline{M}$ such that $e_{n+2} = \eta(p)$. The vector fields
$\widetilde{e}_1, \dots, \widetilde{e}_{n+2}$ form a global orthonormal referential of $T \overline{M}$, so that we can write 
\begin{align} 
\eta = \sum_{i=1}^{n+2} a_i \widetilde{e}_i \label{normal}
\end{align}
for certain functions $a_i \in C^{\infty}(M)$. Notice that $a_i(p) = 0$ for $i \in \{1, \dots, n + 1 \}$ and $a_{n+2}(p) = 1$.

For $y \in M$ we have
\begin{align*}
\gamma(y) = \Gamma_y(\eta(y)) = \Gamma_y \left(\sum_{i=1}^{n+2} a_i(y) \widetilde{e}_i (y) \right) = \sum_{i=1}^{n+2} a_i(y) \Gamma_p(e_i).
\end{align*}
Therefore, if $X \in T_p M$,
\begin{align} 
\Gamma_p^{-1}(D \gamma(p) \cdot X) = \Gamma_p^{-1} \left( \sum_{i=1}^{n+2} X(a_i) \Gamma_p(e_i) \right) = \sum_{i=1}^{n+2} X(a_i) e_i. \label{composta} 
\end{align} 
From (\ref{normal}) and (\ref{composta}) we obtain
\begin{align*}
-A_p(X) &= \overline{\nabla}_X \eta
         = \sum_{i=1}^{n+2} \overline{\nabla}_X(a_i \widetilde{e}_i)
         = \sum_{i=1}^{n+2} \left[ a_i(p) \overline{\nabla}_X \widetilde{e}_i + X(a_i) \widetilde{e}_i (p) \right]  \\ 
        &= \overline{\nabla}_X \widetilde{e}_{n+2} + \sum_{i=1}^{n+2} X(a_i) e_i = \alpha_p(X) + \Gamma_p^{-1}(D \gamma(p) \cdot X),
\end{align*}
which gives the desired result.
\end{proof}

\section{Proof of the main theorem} \label{The main theorem}

Assume from now on that the Euler characteristic of $M$ is zero, so that it admits nonvanishing vector fields. Let $v$ be such a vector field, with unit norm. 
For $t > 0$, define a map $\varphi_t^v : M \to \mathbb{S}^{n+1}(\sqrt{1+t^2})$ by
\begin{align*}
\varphi_t^v(p) = \Gamma_p \left( \eta(p) + t v(p) \right), \quad p \in M.
\end{align*}
The degree formula gives
\begin{align}
\int_M \left( \varphi_t^v \right)^* \omega = \deg \left( \varphi_t^v \right) \int_{\mathbb{S}^{n+1}(\sqrt{1+t^2})} \omega 
= \deg \left( \varphi_t^v \right) \vol (\mathbb{S}^{n+1}) \left( \sqrt{1 + t^2} \right)^{n+1}, \label{degree formula}
\end{align}
where $\omega$ is the volume form of $\mathbb{S}^{n+1}(\sqrt{1 + t^2})$. But $\left( \varphi_t^v \right)^* \omega = \det(D \varphi_t^v) \omega_M$, for $\omega_M$
the volume form of $M$. In the sequel, we will calculate this determinant.

\begin{lem} Let $\overline{v} : M \to \mathbb{S}^{n+1}$ be defined by $\overline{v}(p) = \Gamma_p(v(p))$. Then, the following formula holds
for every point $p \in M$ and $X \in T_p M$:
\begin{align*}
\Gamma_p^{-1} \left( D \overline{v}(p) \cdot X \right) &= \overline{\nabla}_X v - \overline{\nabla}_X \widetilde{v(p)} \\
&= \nabla_X v  + \langle v(p), A_p(X) \rangle \eta(p) - \overline{\nabla}_X \widetilde{v(p)}
\end{align*}
\end{lem}
\begin{proof}
Imitate the proof of Proposition \ref{relationship} with an orthonormal basis $\{ e_1, \dots, e_{n+2} \}$ of $T_p \overline{M}$ such that $e_{n+1} = v(p)$ and
$e_{n+2} = \eta(p)$.
\end{proof}

From the Lemma and Proposition \ref{relationship}, it follows that
\begin{align*}
\Gamma_p^{-1} \left( D \varphi_t^v(p) \cdot X \right) =
- \left( A_p(X) + \alpha_p(X) \right) + t \left[ \nabla_X v  + \langle v(p), A_p(X) \rangle \eta(p) - \overline{\nabla}_X \widetilde{v(p)} \right].
\end{align*}

Consider an orthonormal basis $\{ e_1, \dots, e_n, e_{n+1} = v(p) \}$ of $T_p M$. Defining $u$ by
\begin{align*}
u = \frac{v(p)}{\sqrt{1 + t^2}} - \frac{t \eta(p)}{\sqrt{1 + t^2}} \, ,
\end{align*}
one can directly check that $\{ \Gamma(e_1), \dots, \Gamma(e_n), \Gamma(u) \}$ is an orthonormal basis of $\mathbb{S}^{n+1}(\sqrt{1 + t^2})$. We will express 
$D \varphi_t^v(p)$ relative to these basis.

As a matter of convention, uppercase indices will run from $1$ to $n+1$, whereas lowercase ones will vary from $1$ to $n$. This said, we define the following 
quantities:

\begin{minipage}{.34\linewidth}
  \centering
  \begin{align*}
  h_{AB} &= \langle A_p(e_B), e_A \rangle \\
  \alpha_{AB} &= \langle \alpha_p(e_B), e_A \rangle
  \end{align*}
\end{minipage}%
\begin{minipage}{.33\linewidth}
  \centering 
  \begin{align*}
  a_{ij} &= \langle \nabla_{e_j} v, e_i \rangle \\
  \tilde{a}_{ij} &= \langle \overline{\nabla}_{e_j} \widetilde{v(p)}, e_i \rangle 
  \end{align*}
\end{minipage}%
\begin{minipage}{.34\linewidth} 
  \centering
  \begin{align*}
  v_i &= \langle \nabla_{v(p)} v, e_i \rangle \\
  \tilde{v}_i &= \langle \overline{\nabla}_{v(p)} \widetilde{v(p)}, e_i \rangle 
  \end{align*}
\end{minipage}
\vspace{2mm}
Since $\Gamma_p$ is an isometry, one can show that (see also \cite{brito2016})
\begin{align*}
\langle D \varphi_t^v(p) \cdot e_j, \Gamma_p(e_i) \rangle &=
-(h_{ij} + \alpha_{ij}) + t(a_{ij} - \tilde{a}_{ij}) \\
\langle D \varphi_t^v(p) \cdot e_j, \Gamma_p(u) \rangle &= 
- \sqrt{1 + t^2} \left( h_{n+1,j} + \alpha_{n+1,j} \right) \\
\langle D \varphi_t^v(p) \cdot v(p), \Gamma_p(e_i) \rangle &= 
-(h_{i,n+1} + \alpha_{i,n+1}) + t(v_i - \tilde{v}_i) \\
\langle D \varphi_t^v(p) \cdot v(p), \Gamma_p(u) \rangle &= 
- \sqrt{1 + t^2} \left( h_{n+1,n+1} + \alpha_{n+1,n+1} \right)
\end{align*}
Defining line vectors by
\begin{align*}
V_i &= \left( a_{i1} - \tilde{a}_{i1}, \dots, a_{i,n} - \tilde{a}_{i,n}, v_i - \tilde{v}_i \right) \\
H_A &= \left( h_{A1} + \alpha_{A1}, \dots, h_{A,n+1} + \alpha_{A,n+1} \right), \\
\end{align*}
the matrix of $D \varphi_t^v(p)$ is, then,
\begin{align*}
\renewcommand{\arraystretch}{1.5} 
D \varphi_t^v(p) =
\left[
\begin{array}{c} 
-H_1 + t V_1 \\
\vdots \\
-H_n + t V_n \\
-\sqrt{1 + t^2} H_{n+1}
\end{array}
\right]
\end{align*}
The multilinearity of the determinant gives
\begin{align*}
\det \left( D \varphi_t^v(p) \right) = \sqrt{1 + t^2} \sum_{k = 0}^n \mu_k t^k,
\end{align*}
where 
\begin{align*}
\mu_0 &= (-1)^{n+1} \det \left( H_1, \dots, H_n, H_{n+1} \right) \\
\mu_1 &= (-1)^{n} \sum_{1 \leq i \leq n} \det \left( H_1, \dots, V_i, \dots, H_n, H_{n+1} \right) \\
\mu_2 &= (-1)^{n-1} \sum_{1 \leq i < j \leq n} \det \left( H_1, \dots, V_i, \dots, V_j, \dots, H_n, H_{n+1} \right) \\
&\vdots \\
\mu_n &= - \det \left( V_1, \dots, V_n, H_{n+1} \right).
\end{align*}
From this and (\ref{degree formula}), we conclude that
\begin{align*}
\sum_{k=0}^n t^k \int_M \mu_k = 
\begin{cases}
\vspace{0.2cm} \deg \left( \varphi_t^v \right) \vol(\mathbb{S}^{n+1}) \sum_{k=0}^{n/2} \binom{n/2}{k} t^{2k}, &\quad \text{if } n \text{ is even} \\ 
 \deg \left( \varphi_t^v \right) \vol(\mathbb{S}^{n+1}) \sqrt{1 + t^2} \sum_{k=0}^{(n-1)/2} \binom{(n-1)/2}{k} t^{2k}, &\quad \text{if } n \text{ is odd.} 
\end{cases}
\end{align*}

When $n$ is even, both sides of the above equation are polynomials in $t$, and the powers in the right hand side are all even, which implies the coefficients 
multiplying odd powers in the left hand side all vanish. When $n$ is odd, the presence of the factor $\sqrt{1 + t^2}$ forces all coefficients in the left hand side to be
zero. Furthermore, notice that $\deg \left( \varphi_t^v \right) = \deg(\gamma)$. To see why this is true, let $i$ and $i_t$ be the inclusions of $\mathbb{S}^{n+1}$ and 
$\mathbb{S}^{n+1}(\sqrt{1 + t^2})$ into $V$. The maps $i_t \circ \varphi_t^v$ and $i \circ \gamma$ are (linearly) homotopic, so that 
$\deg(i_t) \deg \left( \varphi_t^v \right) = \deg(i) \deg(\gamma)$. But the degrees of both $i$ and $i_t$ are obviously equal to $1$. Hence, 
\begin{align*}
\int_M \mu_k = 
\begin{cases}
\deg(\gamma) \vol(\mathbb{S}^{n+1}) \binom{n/2}{k/2}, &\quad \text{if } n \text{ and } k \text{ are even} \\ 
0, &\quad \text{if } n \text{ or } k \text{ is odd}  
\end{cases} 
\end{align*}
completing the proof.

\section{Foliations of hypersurfaces} \label{foliations}

Firstly, recall that if $L^l$ is immersed in $\overline{M}^{n+2}$ with codimension $n+2-l \geq 1$ and $N$ is a unit normal to $L$ at a point $p \in L$, the shape operator
of $L$ at this point in the direction of $N$ is the linear operator $A_{N} : T_p L \to T_p L$ given by
\begin{align*}
A_{N}(X) = - \left( \overline{\nabla}_X \eta \right)^{\top}, \quad X \in T_p L,
\end{align*}
where $\eta$ is a normal and unitary extension of $N$ in a neighbourhood of $p$ and $(\cdot)^{\top}$ indicates projection onto $T_p L$. Inspired by this, 
we extend Definition \ref{invariant shape operator}.

\begin{definition} The invariant shape operator of $L$ at the point $p$ in the direction of $N$ is the linear operator $\alpha_N : T_p L \to T_p L$ given by
\begin{align*}
\alpha_N(X) = \left( \overline{\nabla}_X \widetilde{N} \right)^{\top}, \quad X \in T_p L.
\end{align*}
\end{definition}

\noindent Notice that $\overline{\nabla}_X \widetilde{N} \in \{ N \}^{\perp}$, so that this really coincides with the previous definition when the codimension of $L$ is $1$.

Now, let $\mathcal{F}$ be a transversely oriented codimension-one foliation of the compact, connected and oriented immersed hypersurface $M^{2n+1}$ of $\overline{M}^{2n+2}$.
Consider a unit vector field $v \in \mathfrak{X}(M)$ normal to the leaves of $\mathcal{F}$. In this case, the matrices $\left( -a_{ij} \right)$ and 
$\left( \tilde{a}_{ij} \right)$ at a point $p$ are the matrices of $A_{v(p)}$ and $\alpha_{v(p)}$ --- the shape and invariant shape operators of the leaf $L$ of
$\mathcal{F}$ containing $p$ --- with respect to a chosen basis of $T_p L$. Furthermore,
\begin{align*}
\renewcommand{\arraystretch}{1.5}
\mu_{2n} = - \det \left[
\begin{array}{c|c}
                                &   v_1 - \tilde{v}_1        \\
a_{ij} - \tilde{a}_{ij}         &   \vdots                   \\
                                &   v_{2n} - \tilde{v}_{2n}  \\
\hline
h_{2n+1,1} \cdots h_{2n+1,2n}  &  h_{2n+1,2n+1}  
\end{array} 
\right].
\end{align*}
If $\mu_{2n} = 0$, then $\deg(\gamma) = 0$ due to (\ref{curvature integrals}). This, in turn, implies that $M$ itself is parallelisable. Indeed, consider the vector bundle 
map $\widetilde{\gamma}(p, v) = (\gamma(p), \Gamma_p(v))$, which covers $\gamma$:
\begin{align*}
\begin{CD}
TM @>\tilde{\gamma}>> T \mathbb{S}^{2n+1}\\
@VVV @VVV\\
M @>\gamma>> \mathbb{S}^{2n+1}
\end{CD}
\end{align*}
Notice that $TM$ is the pullback of $T \mathbb{S}^{2n+1}$ by $\gamma$. If the degree of $\gamma$ is zero, then $TM$ is trivial, since homotopic maps induce isomorphic pullback 
bundles. Thus, we proved:

\begin{thm} \label{thmfoliations} Let $\mathcal{F}$ be a transversely orientable codimension-one foliation of the compact, connected and oriented immersed hypersurface $M^{2n+1}$ of
$\overline{M}^{2n+2}$, and let $v$ a unit length vector field tangent to $M$ and normal to the leaves of $\mathcal{F}$. If the operator $A_v + \alpha_v$ has rank 
less than or equal to $2(n-1)$ along the leaves, then the degree of the Gauss map $\gamma : M \to \mathbb{S}^{2n+1}$ is equal to zero. In particular, $M$ is parallelisable.
\end{thm} 

\begin{remark} Instead of a foliation, we may assume only a transversely oriented codimension-one distribution has been given (not necessarily integrable). The 
hypothesis on the rank of $A_v + \alpha_v$ is the same, but now these operators lack some geometrical meaning.
\end{remark}

\begin{cor} \label{Lie group} Let $G^{2n+2}$ be a Lie group with a left invariant metric and equipped with the translational structure of Example \ref{Left translation}. 
Consider a compact, connected and oriented immersed hypersurface $M^{2n+1}$ of $G^{2n+2}$ together with a transversely orientable codimension-one foliation $\mathcal{F}$ 
and a unit length vector field $v$ tangent to $M$ and normal to the leaves of $\mathcal{F}$. If $G$ is commutative and the rank of the shape operator of the leaves is
smaller than or equal to $2(n-1)$, then the degree of the Gauss map $\gamma : M \to \mathbb{S}^{2n+1}$ is zero. In particular, $M$ is parallelisable.
\end{cor}
\begin{proof}
Given left invariant vector fields $X, Y, Z$ on $G$, the Koszul formula gives
\begin{align*}
2 \langle \overline{\nabla}_X Y, Z\rangle &= X \langle Y,Z \rangle + Y \langle X,Z \rangle - Z \langle X,Y \rangle + \langle [X,Y],Z \rangle 
- \langle [X,Z],Y \rangle - \langle [Y,Z],X \rangle \\
&= X \langle Y,Z \rangle + Y \langle X,Z \rangle - Z \langle X,Y \rangle \\
&= 0,
\end{align*}
since $G$ is commutative and the inner product of left invariant vector fields is constant along $G$. Hence, $\alpha$ is identically zero. The result then 
follows from Theorem \ref{thmfoliations}.
\end{proof} 

\begin{cor} Let $G$ and $M$ be as in Corollary \ref{Lie group}. If $M$ is \emph{not} parallelisable, then it does not admit transversely orientable codimensionone 
foliations with totally geodesic leaves.
\end{cor}

Theorem \ref{thmfoliations} asserts that for a hypersurface in a translational manifold, the existence of distributions with certain nullity relates directly to the degree of its normal map. The problem of finding a Riemannian metric with respect to which a given codimension-one foliation is totally geodesic was completely solved by Ghys \cite{ghys}. Theorem \ref{thmfoliations} and its consequences make a small contribution in the light of this general situation. 

On the other hand, Euclidean hypersurfaces carrying a codimension-one totally geodesic foliation have been classified by Dajczer {\it et al} in \cite{daj}. We hope that combining the extrinsic geometry of distributions and foliations on hypersurfaces to the topology of its immersion will bring new ideas and insights towards a general framework.

\end{document}